\documentclass[12pt,letterpaper,titlepage]{amsart}
\usepackage{amsmath, amssymb, amsthm, amsfonts,amscd, amsaddr}

\theoremstyle{plain}
\newtheorem{theorem}{Theorem}[section]
\newtheorem{proposition}[theorem]{Proposition}

\newtheorem{lemma}[theorem]{Lemma}

\theoremstyle{definition}

\newtheorem{rmk}[theorem]{Remark}
\numberwithin{equation}{section}

\newcommand{\bs}{\backslash}
\newcommand{\cc}{\mathcal{C}}
\newcommand{\C}{\mathbb{C}}

\newcommand{\Hb}{\mathbb{H}}

\newcommand{\Oc}{\mathcal{O}}
\newcommand{\R}{\mathbb{R}}

\newcommand{\PP}{\mathbb{P}}
\newcommand{\bO}{\mathbb{O}}

\newcommand{\Sl}{\operatorname{SL}}
\newcommand{\SO}{\operatorname{SO}}
\newcommand{\Spin}{\operatorname{Spin}}

\newcommand{\SL}{\operatorname{SL}}
\newcommand{\SP}{\operatorname{Sp}}
\newcommand{\Sp}{\operatorname{Sp}}
\newcommand{\SU}{\operatorname{SU}}

\newcommand{\Ad}{\operatorname{Ad}}
\newcommand{\Int}{\operatorname{c}}
\newcommand{\ad}{\operatorname{ad}}
\newcommand{\diag}{\operatorname{diag}}

\def\af{\mathfrak{a}}

\def\gf{\mathfrak{g}}
\def\ff{\mathfrak{f}}

\def\hf{\mathfrak{h}}
\def\kf{\mathfrak{k}}
\def\lf{\mathfrak{l}}
\def\mf{\mathfrak{m}}
\def\nf{\mathfrak{n}}

\def\pf{\mathfrak{p}}
\def\qf{\mathfrak{q}}

\def\sf{\mathfrak{s}}

\def\so{\mathfrak{so}}
\def\sp{\mathfrak{sp}}

\def\su{\mathfrak{su}}
\def\uf{\mathfrak{u}}

\def\1{{\bf1}}

\def\oline{\overline}

\newcommand{\hcpt}{\hf'_c}
\newcommand{\hnon}{\hf'_n}

\begin{document}
\title[Real flag manifolds]
{Finite orbit decomposition of real flag manifolds}
\date{August 18, 2014}
\begin{abstract}
Let $G$ be a connected real semi-simple Lie group and $H$ a closed
connected subgroup. Let $P$ be a minimal parabolic subgroup of $G$.
It is shown that $H$ has an open orbit on
the flag manifold $G/P$
if and only if it has finitely many orbits on $G/P$.
This confirms a conjecture by T.~Matsuki.
\end{abstract}
\author[Kr\"otz]{Bernhard Kr\"otz}
\email{bkroetz@math.uni-paderborn.de}
\address{Universit\"at Paderborn\\Institut f\"ur Mathematik\\Warburger Stra\ss e 100\\
D-33098 Paderborn\\Germany}
\author[Schlichtkrull]{Henrik Schlichtkrull}
\email{schlicht@math.ku.dk}
\address{University of Copenhagen\\Department of Mathematics\\Universitetsparken 5\\
DK-2100 Copenhagen \O\\Denmark}
\thanks{The first author was supported by ERC Advanced Investigators Grant HARG 268105}
\maketitle

\section{Introduction}

Let $G$ be a connected real semi-simple Lie group and $P$
a minimal parabolic subgroup. Let $H <G$ be a closed 
and connected subgroup. The following theorem was 
conjectured by T.~Matsuki in~\cite{Mat2}.

\begin{theorem}\label{main thm}
If there exists an open $H$-orbit on 
the real flag variety $G/ P$ then
the double coset space $H\backslash G / P$
is finite. 
\end{theorem}

The purpose of this paper is to give a proof of Matsuki's conjecture. 
Note that the converse statement is easy: if $H\bs G / P$ is finite, 
then at least one double coset must be open as a consequence of the Baire category theorem. 
Further we remark that the theorem becomes false if the parabolic subgroup
$P$ is not minimal. A standard counterexample is $G=\Sl(3,\R)$ with $P$ a maximal parabolic and 
$H$ the unipotent part of $P$.  

In case $G$ is a complex algebraic reductive group, the minimal parabolic $P$
equals a Borel subgroup $B$ of $G$. 
A complex algebraic subgroup which has an open orbit on $G/B$
is called {\it spherical}. In this case
the finiteness of $H\bs G/B$ for a spherical subgroup $H$ is
a result of Brion \cite{Br} and
Vinberg \cite{Vinberg} with a simplified proof by Knop \cite{Knop}.
The spherical subgroups of a complex algebraic group have been classified
by Kr\"amer \cite{Kraemer} and Brion \cite{Brion}, but to our knowledge
there exists no such classification for $G$ real.

For $G$ real and $H$ a symmetric subgroup (that is, it is
the identity component of the set of fixed points for an involution),
it was shown by Wolf \cite{Wolf} that the conclusion (and hence the assumption) of 
Theorem \ref{main thm} is always fulfilled.

\par Our proof of Theorem \ref{main thm} proceeds in two steps. In the first step
we reduce the assertion to the case where the real rank of $G$ is one. The argument
for this step is essentially due to Matsuki (cf.~\cite{Mat2}).
For rank one groups we then treat the cases where $H$ is reductive or non-reductive 
in $G$ separately. In case $H$ is non-reductive, one shows that $H$ is contained 
in a conjugate of $P$ and that there are 2, 3 or 4  $H$-orbits on $G/P$.
For reductive $H$ we prove a refined statement: 

\begin{theorem}\label{main thm2} Suppose that $G$ is of real rank one
and that $H$ is a connected reductive 
subgroup with an open orbit on $G/P$. Then there is a symmetric subgroup 
$H'\supset H$ such that the $H'$-orbit decomposition of $G/P$ 
equals the $H$-orbit decomposition. 
\end{theorem}

This concludes the proof of Theorem \ref{main thm} since as mentioned
above, $H\bs G/P$ is finite for all symmetric subgroups of $G$.

Finally we remark that the conclusion of Theorem \ref{main thm2} is false in
higher real rank. For example, $H=\SL(2,\R)$ diagonally embedded in the triple product 
$G=\SL(2,\R)^3$ admits an open orbit in $G/P$ (see \cite{DKS}).
Let $P=P_1^3$, where $P_1$ is a parabolic subgroup of $\SL(2,\R)$, then
the $H$-orbit through the origin of $G/P$ is one-dimensional. On the other hand, the proper symmetric subgroups containing $H$ 
have the form $H'=\SL(2,\R)^2$, embedded by $(x,y)\mapsto(x,x,y)$ 
up to permutation, and for these groups the orbit through the 
origin is two-dimensional.

After we finished this paper, it was brought to our attention by T.~Kobayashi
that the subject was previously considered by F. Bien in \cite{Bien}, where an outline is given
for a proof of Theorem \ref{main thm}. Apart from the reduction suggested by Matsuki, 
there is however no overlap with the current approach.

\section{Reduction to the rank one case}

Let us call the pair $(G,H)$ {\it real spherical} provided there are 
open $H$-orbits on $G/P$.  This means that the corresponding infinitesimal 
pair $(\gf, \hf)$ of Lie algebras satisfies $\gf=\hf +\Ad(x)\pf$
for some $x\in G$. Since all our groups will be real, we will just say
{\it spherical}. As customary we denote Lie subgroups of $G$ 
by upper case roman letters and their Lie algebras by the corresponding lower 
case German letters. 

\par Let $(G,H)$ be a spherical pair. Matsuki remarked in \cite{Mat2}, p.~813, that Theorem \ref{main thm} holds true provided it 
is valid for all spherical pairs $(G,H)$ where $G$ is a
semisimple Lie group of real rank one. 
The purpose of this section is to provide a proof of this remark. We follow
closely the proof of Theorem 4 in \cite{Mat2}. 

We assume that the assertion of Theorem \ref{main thm}
holds for all rank one groups. Let $(G,H)$ be a spherical pair,
then there exists an open orbit in $G/P$, say $Hx_0P$.
 
We fix a Cartan 
decomposition $\gf=\kf\oplus\sf$,
with corresponding involution $\theta$, and a maximal abelian
subspace $\af\subset\sf$. We assume that $P\supset A$ and denote by
$\Pi\subset \af^*$ the set of simple roots 
attached to $P$. For $\alpha\in \Pi$ we define the parabolic subgroup
$P_{\alpha} := Ps_\alpha P \cup P$ where $s_\alpha\in G$ is a Weyl group 
representative of the reflection associated to $\alpha$, and write
$P_\alpha=L_\alpha U_\alpha$ for its Levi decomposition relative to
$A$. Then $L_\alpha$ has real rank one.
Write 
$G=P_{\alpha_1} \cdot \ldots \cdot P_{\alpha_n}$ as a product of such
parabolics. Set 
$$P^i:= P_{\alpha_1} \cdot \ldots\cdot P_{ \alpha_i} \qquad (0\leq i \leq n)$$
with the convention that $P^0=P$. 
We will prove by induction on $i$ that $H\bs H x_0 P^i / P $ is finite. 
The
theorem is reached after $n$ steps.
 
\par The case of $i=0$ is clear. Assume that 
\begin{equation} \label{finite}Hx_0P^i = H g_1 P \cup \ldots \cup H g_k P
\end{equation}
for elements $g_1, \ldots, g_k \in G$ and let $\alpha=\alpha_{i+1}\in \Pi$.    
Then it is sufficient to show for any $g\in Hx_0P^i$ that 
$H g P_ \alpha$ breaks into finitely 
many $H\times P$-orbits.

We shall first prove that there exists a
relatively open $H\times P$-orbit in
$H g P_ \alpha$. More precisely, we will show
that for some $r=1,\dots,k$ we have 
$Hg_r P\subset H g P_\alpha$ open.

\par Note that $\pf_\alpha= \pf + \gf^{-\alpha}+\gf^{-2\alpha}$ and 
set $V:=\exp(\gf^{-\alpha}+\gf^{-2\alpha})$. Then $VP$ is an open 
neighborhood of $\1$ in $P_\alpha$. 
As $Hx_0P^i$ is open and contains $g$,
we obtain an open  subset $O\subset V$ with $\1\in O$ 
such that $gx\in Hx_0P^i$ for all $x\in O$.
Moreover, because of (\ref{finite}) we have 
$O=\bigcup_{r=1}^k O_r$ with 
$$O_r=\{ x\in O\mid gx \in Hg_rP\}\, ,$$
and at least one of these sets
has non-empty interior by Baire's theorem.
Fix such an $r$, then
$HgxP=Hg_rP$ for every $x\in O_r,$
and as $O_rP$ has non-empty interior in $VP$
it follows that $Hg_rP$ has non-empty interior
in $HgP_\alpha$, hence is an open subset
by transitivity of  $H\times P$.

\par We can now show that $HgP_\alpha$
decomposes into finitely many orbits. 
Notice that we have
$HgP_\alpha=Hg_rP_\alpha$.
For simplicity we replace $H$ by $g_rHg_r^{-1}$,
and claim that if $HP$ is open in $HP_\alpha$
then the latter set is a finite
union of $H\times P$ orbits.

We write $\pf_\alpha = \lf_\alpha + \uf_\alpha$ for the Levi decomposition 
of the Lie algebra of $P_\alpha$. Further we denote by 
$$\pi_\alpha: \pf_\alpha \to \lf_\alpha$$
the projection along $\uf_\alpha$ and remark that the map  is 
a Lie algebra homomorphism. 
Set $\hf_\alpha:= \pi_\alpha(\pf_\alpha\cap \hf)$. As $HP$ is open we find
$$\hf+\pf=\hf+\pf_\alpha$$ and hence $\pf_\alpha=(\pf_\alpha\cap\hf)+\pf$.
In turn this implies that 
$$\lf_\alpha= \hf_\alpha + (\lf_\alpha \cap \pf).$$ In other words, 
$(\lf_\alpha, \hf_\alpha)$ is a rank-one spherical pair.  
We thus get that $H_\alpha\bs L_\alpha/ 
(L_\alpha\cap P)$ is finite 
(the fact that $L_\alpha$ can be non-connected
does not matter, because all its components intersect
non-trivially with $P$).
We write $$L_\alpha=\cup_{j=1}^m H_\alpha x_j (L_\alpha\cap P)$$
and claim that $$HP_\alpha=
\cup_{j=1}^m H x_j P.$$
As $P_\alpha=L_\alpha P$, it suffices to show that
$hx_j\in Hx_jP$ for all $h\in H_\alpha$ and all $j$.
Note that $\hf_\alpha$ is contained in the subalgebra
$(\pf_\alpha\cap\hf)+\uf_\alpha$ of $\pf_\alpha$, and
hence $H_\alpha$ is contained in the subgroup
$(P_\alpha\cap H)U_\alpha$ of $P_\alpha$. It follows that
$$hx_j\in (P_\alpha\cap H)U_\alpha x_j=(P_\alpha\cap H)x_j U_\alpha
\subset Hx_jP$$
as claimed.

Hence the proof of Theorem \ref{main thm} is reduced to 
the following result.

\begin{proposition}\label{th1} Let $G$ be a semisimple Lie group of real rank one and
$H$ a connected spherical subgroup. Then the number of $H$-orbits on $G/P$
is finite.
\end{proposition}

\noindent{\it Example: $G=\Sl(2,\R)$.}
Every one-dimensional subalgebra
is conjugate to $\kf$, $\af$ or $\nf$.
The first two are symmetric, and in the third case finiteness
of $H\bs G/P$ follows from the Bruhat decomposition. Hence $H\bs G/P$
is finite for every non-trivial connected subgroup $H$.

\subsection{Simple groups}\label{simple}

Once the real rank is one we can easily reduce to the case that $G$ is simple. 
Otherwise $G$ is locally
isomorphic to $G_1\times K_2$ where $G_1$ is simple of real rank one and
$K_2$ is compact. Then $P=P_1\times K_2$ where $P_1\subset G_1$
is minimal parabolic, and hence $G/P=G_1/P_1$. Moreover,
if $H_1$ denotes the projection of $H$ on $G_1$,
then $H$-orbits on $G/P$ are the same as 
$H_1$-orbits on $G_1/P_1$.

\section{Non-reductive spherical subgroups}

In this section we prove Proposition \ref{th1} for
spherical subgroups which are not reductive.

\begin{lemma}\label{H in P} 
Let  $\gf$ be a real reductive Lie algebra of real rank one  and $\hf$ a subalgebra,
which is not reductive in $\gf$. 
Then $\hf\subset\pf$, up to conjugation.
\end{lemma}

\begin{proof} This is an easy consequence of the main result of \cite{BorelTits}.
Let $\uf$ denote the unipotent radical of $\hf$, 
then  $\uf\neq 0$ by assumption. According to \cite{BorelTits}
there is a parabolic subalgebra $\pf'\subset\gf$ such that
\begin{enumerate}
\item the normalizer of $\uf$ is contained in $\pf'$, and
\item the unipotent radical of $\pf'$ contains $\uf$. 
\end{enumerate}
Hence $\hf\subset\pf'\neq\gf$, and $\pf'$ is conjugate to  $\pf$ because of rank one.
\end{proof}

\par For a rank one group $G$ we let $s$ be a Weyl group representative and recall the 
Bruhat decomposition $G=P\cup PsP$. The following result
together with Lemma \ref{H in P} 
implies the conclusion of Proposition \ref{th1} for $H$ non-reductive

\begin{lemma}\label{reduction} Let $\gf$ be a simple real rank one Lie algebra, $\hf$ a spherical subalgebra,
and $\pf$ a minimal parabolic subalgebra for which $\hf\subset \pf$. 
Then the Bruhat decomposition $G/P = P \cup PsP$ is $H$-stable and there are at most four orbits of $H$ on $G/P$. 
\end{lemma} 

\begin{proof} 
We denote by $P=MAN$ the standard Langlands decomposition
of $P$ relative to $A$.

The cells of the Bruhat decomposition are $P$-stable, 
hence also $H$-stable.
In particular, the closed cell $P\in G/P$ is an $H$-orbit.
Hence by assumption the open cell $\Oc:=NsP$ admits
at least one open $H$-orbit, and the assertion is
that then it decomposes into at most three $H$-orbits.

We decompose $\hf=\lf \ltimes \nf_1$ with $\lf$ reductive in $\gf$ and 
$\nf_1$ an ideal which acts on $\gf$ nilpotently. 
As $\hf\subset\pf$ we have $\nf_1\subset\nf$, and 
it is no loss of generality 
to assume that $\lf\subset \mf +\af$.  
Then with $\mf_1= \lf \cap \mf$ we have  
$$\hf= \mf_1 + \R X +\nf_1$$ for some $X=Y+Z$ with 
$Y \in \mf$ and $Z\in\af$. Since $[\lf,\lf]\subset\mf_1$, 
the element $X$ belongs to the center of $\lf$ and commutes with $\mf_1$.
Let $\sf_1=\R X$, then both $\sf_1$ and $\mf_1$ normalize $\nf_1$ and we have
$$H=M_1S_1N_1$$
for the corresponding subgroups of $G$. Since
$M_1S_1s\subset sP$
we see that the $H$-orbits in $NsP$
are the sets $N_1\Int_{M_1S_1}(x)sP$ for $x\in N$,
where $\Int_g(x)=gxg^{-1}$.
In particular, $N_1sP\subset NsP$ is an $H$-orbit.
If $N_1=N$ we are done, hence we may assume $\nf_1\subsetneq \nf$.

Note that $Z\neq 0$. Otherwise $X=Y$, hence $X\in\mf_1$ and $\hf= \mf_1+\nf_1$.
Then $H=M_1 N_1$ and $H$-orbits 
have the form $N_1\Int_{M_1}(x)sP$.
Since $N_1\Int_{M_1}(x)$ cannot be open in $N$ for any $x$, a contradiction is reached. 

The element $X$ acts semisimply on $\nf_\C$ and preserves
the subspace $\nf_{1\C}$. We denote by $\alpha$ the indivisible positive root of $\af$, then
$\nf=\gf^{\alpha}\oplus \gf^{2\alpha}$ and $[\nf,\nf]\subset \gf^{2\alpha}$. 
Since $Z\neq 0$ the spaces
$\gf^{\alpha}_\C$ and $\gf^{2\alpha}_\C$ have no eigenvalues
in common for $\ad(X)$. It follows that
\begin{equation}\label{grading1}
\nf_1=(\nf_1 \cap\gf^{\alpha})\oplus (\nf_1\cap\gf^{2\alpha}).
\end{equation}
Let $\nf_0$ be the orthogonal complement of $\nf_1$ in $\nf$, then $\nf_0\neq 0$ and
(\ref{grading1}) implies
\begin{equation}\label{grading0}
\nf_0=(\nf_0 \cap\gf^{\alpha})\oplus (\nf_0\cap\gf^{2\alpha}).
\end{equation}
It follows from (\ref{grading1}) and (\ref{grading0}) together with
\cite{Helgason} Ch.~IV, Lemma 6.8, that the exponential map
induces a diffeomorphism of $\nf_0$ with the left coset space $N_1\bs N$.
Note that $\nf_0$ is $M_1S_1$-invariant. 
We conclude that the $H$-orbits in $NsP$ correspond to the 
orbits of $\Ad(M_1S_1)$ on $\nf_0$. In particular, the $H$-orbit $N_1sP$ corresponds 
to $\{0\}\subset\nf_0$.

Since $M$ acts isometrically on $\nf$ it follows that $M_1$ acts isometrically
on $\nf_0$. Furthermore, let $X=Y+Z$ be normalized such that $\alpha(Z)=1$ and put
$s_t=\exp(tX)\in S_1$, then for $j=1,2$
\begin{equation}\label{dilation}
\|\Ad(s_t)x\|=e^{jt}\|x\|, \qquad x\in \nf_0 \cap\gf^{j\alpha}, t\in\R.
\end{equation}
It follows that if $\Oc_1\neq \{0\}$ is an $\Ad(M_1S_1)$-orbit in $\nf_0$,  
then the intersection of $\Oc_1$ with every sphere in $\nf_0$ 
is non-empty and is an $\Ad(M_1)$-orbit.

Assume first that $\dim\nf_0>1$. Then spheres in $\nf_0$ are connected, and
by compactness an open $\Ad(M_1)$-orbit is 
the entire sphere. Hence we conclude that the open $\Ad(M_1S_1)$-orbit in $\nf_0$ 
is $\nf_0\setminus \{0\}$. In this case
$NsP$ decomposes in two orbits, $N_1sP$ and its complement.

Assume finally that $\dim\nf_0=1$. In this case 
it follows from (\ref{dilation}) that $Ad(s_t)x=e^{jt}x$ for all $x\in\nf_0$ 
and all $t\in\R$
(where $j=1$ or $2$). Hence in this case there are three orbits in $\nf_0$, corresponding to
$\{0\}$ and the two components of its complement.
\end{proof}

\begin{rmk} The proof shows a bit more.  
The open $N$-orbit $NsP$ breaks into at most 
three $H$-orbits. 
If we identify $N$ with $\R^n$ and $N_1$ with $\R^k$, then these $H$-orbits are of the following type:
\begin{enumerate}
\item $\R^n$ (one orbit, the case where $H\supset N$) when $k=n$. 
\item $\R^k$ and $\R^{n-k}\bs\{0\} \times \R^k$ when $0\leq k <n-1$.  
\item $\R^{n-1}$, $\R^+ \times \R^{n-1}$, $\R^- \times \R^{n-1}$ when $k=n-1$. 
\end{enumerate}
\end{rmk}

\section{Some results in real rank one}\label{sect 4}

We now turn to the case where $G$ has real rank one and our spherical
subgroup $H\subset G$ is reductive, in which case Proposition  \ref{th1}
will ultimately be shown from Theorem  \ref{main thm2}. 
The proof of that theorem will be given after 
we have prepared for it through several sections.
Our first preparation, Proposition \ref{prop},
consists of showing in this case 
that if $H\subset H'\subset G$ with
$H'$ symmetric, then the $H$-orbits in $G/P$ coincide with the $H'$-orbits.
The proof uses Matsuki's description of the $H'$-orbits on $G/P$.

\begin{lemma}\label{decomposition}  
Let $\gf$ be a real reductive Lie algebra with Cartan involution
$\theta$, and let $\gf_n\subset\gf$ be its maximal non-compact ideal.
Let $\hf\subset \gf$ be a $\theta$-stable subalgebra
such that $\gf=\hf+\kf$. Then $\gf_n\subset\hf$. 
\end{lemma}

\begin{proof} It follows from the assumption
that $\sf\subset\hf$ and this implies
the conclusion as $\gf_n$ is generated by $\sf$.
\end{proof}

Recall that a subalgebra $\hf'\subset\gf$ is called
symmetric if it is the fixed points of an involution  
of $\gf$. Recall also that for every involution there
exists a commuting Cartan involution. Given an involution $\sigma$,
we write $\gf=\hf'+\qf$ for the corresponding decomposition of $\gf$.

\begin{lemma}\label{H cap P}
Let $\gf$ be simple of real rank one
and let $\hf'$ be a proper symmetric subalgebra
defined by an involution $\sigma$ commuting with $\theta$.
Let $\pf=\mf+\af+\nf$ be a minimal
parabolic subalgebra with the indicated Langlands decomposition,
and assume $\af\subset\sf\cap\qf$.
Then 
\begin{equation*}
\gf=\hf'+\pf \quad\text{and}\quad\hf'\cap\pf\subset \mf.
\end{equation*} 
\end{lemma}

\begin{proof} 
This
follows from \cite{Mat}, Theorem~3, 
since $\sigma(\nf)=\theta(\nf)$.
\end{proof}

\begin{lemma}\label{exist symmetric hf} 
Let $G$ be a simple Lie group of real rank one 
and let $\hf\subset \hf'\subsetneq \gf$ be 
reductive subalgebras such that 
\begin{enumerate}
\item $\hf$ is spherical
\item $\hf'$ is symmetric and defined by an 
involution commuting with $\theta$.
\end{enumerate}
Then:
\begin{enumerate}\setcounter{enumi}{2}
\item There exists a minimal parabolic subalgebra
$\pf$ with $\af\subset\sf\cap\qf$ 
such that $\hf'=\hf+(\hf'\cap\mf)$ and
$$H'=H(H'\cap M).$$
\item 
Let $\hnon$ be the maximal non-compact ideal in  $\hf'$, then
$\hnon\subset \hf.$
\end{enumerate} 
\end{lemma}

\begin{proof} 
Let $\pf_1$ be a minimal parabolic subalgebra for which
$\hf+\pf_1=\gf$. Then $\hf'+\pf_1=\gf$ and it follows from
\cite{Mat}, Theorems 1 and 3, that $\pf_1$ is
$H'$-conjugate to a minimal parabolic subalgebra
$\pf$ with $\af\subset\sf\cap\qf$. Thus
$\gf=\hf+\Ad(x)\pf$ for some $x\in H'$.
Since $\hf\subset \hf'$ this implies that
$$\hf'=\hf+\hf'\cap\Ad(x)\pf=\hf+\Ad(x)(\hf'\cap\pf).$$
{}From Lemma \ref{H cap P} we find $\hf'\cap\pf=\hf'\cap\mf$,
which is compact. It follows that $Hx(H'\cap M)$ is open and closed in $H'$,
hence equal to $H'$. Hence $H\times(H'\cap M)$ is left$\times$right transitive
on $H'$, and (3) follows.

Since $\hf$ is reductive in $\gf$, it is reductive in $\hf'$.
Hence some $H'$-conjugate $\hf_1$ of it is $\theta$-stable.
The conclusion in (3) is valid for $\hf_1$
and hence $\hf'=\hf_1+\hf'\cap\kf$.
It now follows from  Lemma \ref{decomposition}
that $\hnon\subset\hf_1$.
Since $\hnon$ is an ideal this implies $\hnon\subset\hf$ as well.
\end{proof}

\begin{proposition}\label{prop}
Let $G$ be a connected simple Lie group of real rank one
and let $H\subset H'$ be connected reductive subgroups
such that $H$ is spherical and $H'$ is symmetric and proper in $G$. 
Then $H$ is transitive on each $H'$-orbit in $G/P$.
\end{proposition}

\begin{proof} 
Choose a Cartan involution which commutes with the involution
which defines $H'$.
Let $\pf$ be as in Lemma \ref{exist symmetric hf}.
Since the real rank of $G$ is one, it follows from
Matsuki's orbit description in \cite{Mat} that $H'$ has only 
open and closed orbits in $G/P$. The open orbits are
of the form $H'xP$ for 
$x\in N_K(\af)$, the normalizer in $K$ of $\af$,
and the closed orbits are of the form 
$H'yP$ with $y\in K$ such that $\Ad(y)(\af)\subset \hf'$.

It follows from Lemma \ref{exist symmetric hf} that
$$H'xP=H(H'\cap M)xP=HxP$$ 
for $x\in N_K(\af)$.

Let $\hcpt$ denote the ideal in $\hf'$
which is complementary to $\hnon$. Then $\hf'=\hnon\oplus\hcpt$
and $H'=H'_nH'_c$. If $\Ad(y)(\af)\subset \hf'$ then 
$\Ad(y)(\af)\subset \hf'_n$ and hence $\Ad(y)(\af)$
is centralized by $H'_c$. It follows that 
$H'_c\subset yMy^{-1}$ and hence
$$H'yP=H'_nyP=HyP$$
since $H_n'\subset H$ by Lemma \ref{exist symmetric hf}.
\end{proof}

\section{Example: The Lorentzian groups}\label{SO}

Before we treat the general case, it is instructive to see the proof
of Theorem \ref{main thm2} for the case of $\SO_0(1,n)$ for $n\ge 2$. 

\begin{proof} We observe that 
$G=\SO_0(1,n)$ acts on $\R^{n+1}$.  In the sequel we write the elements 
of $x\in \R^{n+1}$ as $x=(x_0, x')$ with $x'\in \R^n$. The stabilizer $P\subset G$
of the line $\R(1,1,0,\dots,0)\in\PP(\R^{n+1})$ is a minimal parabolic subgroup. 
Note that $G/P= S^{n-1}$ is an $n-1$-dimensional sphere which
we shall identify with the projective quadric: 
\begin{equation} \label{sphere}
S^{n-1} = \{ [x] \in \PP(\R^{n+1})\mid x_0^2 = \|x'\|^2= x_1^2 + \ldots + x_n^2\}.
\end{equation}

Let $\hf$ be a reductive spherical subalgebra, and let $\hf=\hf_n\oplus \hf_c$
be the decomposition of $\hf$ in ideals, such that
$\hf_n$ is non-compact and $\hf_c$ is compact.
Since $\so(1,n)$ has rank one and root multiplicity $m_{2\alpha}=0$,
the same must be true for $\hf_n$. Hence $\hf_n=\so(1,p)$ for 
some $0\leq p \leq n$. Furthermore,
by conjugation of $\hf$ we can arrange that $\hf_n=\so(1,p)$ is realized in 
the left upper corner of $\gf=\so(1,n)$, and accordingly: 
\begin{equation}\label{so-case}
\hf= \hf_n \oplus \hf_c, \quad H=\SO_0(1,p)\times H_c
\end{equation}
with $\hf_c\subset\so(n-p)$ and $\so(n-p)$ embedded in the lower right corner. 
It now follows from Proposition \ref{prop} that
orbits on $G/P$ for $H$ are 
the same as for the symmetric subgroup 
$H'=\SO_0(1,p)\times \SO(n-p)$ of $G$.
Thus the proof of Theorem \ref{main thm2}
is complete for this case.
\end{proof}

\begin{rmk}\label{exist ML} It follows from the above
that every spherical subgroup in $\SO_0(1,n)$ is conjugate
to a subgroup of $H'=\SO_0(1,p)\times \SO(n-p)$
of the form (\ref{so-case})
for some $0\leq p\leq n$. Furthermore since
$H'\cap M=\SO(n-p-1)$ in this case, it follows that such a subgroup
is spherical if and only if $H_c$ is transitive on
$S^{n-p-1}=\SO(n-p)/SO(n-p-1)$.
Besides $H=H'$ this can be attained in case $p$ satisfies 
certain parity conditions.
A typical example is $p=n-2k$ and
$$H=\SO_0(1,n-2k)\times\SU(k),$$
since $H_c=\SU(k)$ acts transitively
on the spheres in $\R^{2k}$.
For $p=n-4k$ we can also take $H_c=\SP(k)$ which again acts transitively 
on spheres. Besides these two series there are 
three exceptional cases (see \cite{Bor} for the classification of transitive 
actions of compact Lie groups on spheres).  
\end{rmk}

\section{Classifications}\label{sect 6}
In this section we prepare for the proof of Theorem \ref{main thm2} 
by showing that in the rank one case a maximal reductive
subgroup is either symmetric or not spherical.
This will be done by applying
some results, which in turn are derived from known classifications 
of simple Lie groups and their subgroups. 

We first recall the classification of the simple real rank one Lie 
algebras
\begin{equation}\label{rank one classification}
\so(1, n),\, \su(1,n), \,\sp(1,n), \, \ff_4\, 
\end{equation}
where $\ff_4=\ff_{4 (-20)}$, the real form of $\ff_{4\C}$
with maximal compact $\so(9)$.
In the first series $n$ is limited to $n\geq 2$. In the 
second and third series $n\geq 1$ is allowed, but as
$\so(1,2)\simeq \su(1,1)$ and $\so(1,4)\simeq \sp(1,1)$ an exhaustive
list is obtained by taking $n\ge 2$ in all cases.

\subsection{Symmetric subalgebras} 

\begin{lemma}\label{Berger list}
The symmetric pairs (excluding $\hf=\gf$ and $\hf=\kf$)
for the simple real rank one Lie algebras are
$$
\begin{alignedat}3
\gf&=\so(1, n),\quad& &\hf_m&&=\so(1,m)\times\so(n-m),\quad 0< m< n, \\
\gf&=\su(1, n),& &\hf_{m}&&=\sf(\uf(1,m)\times\uf(n-m)),\quad 0< m <n, \\
&& &\hf&&=\so(1,n) \\
\gf&=\sp(1, n),& &\hf_{m}&&=\sp(1,m)\times\sp(n-m),\quad 0< m <n ,\\
&& &\hf&&=\uf(1,n) \\
\gf&=\ff_4,& &\hf_1&&=\so(1,8),\quad \hf_2=\sp(1,2)\times \sp(1), 
\end{alignedat}
$$
\end{lemma}

\begin{proof}
This is seen from Berger's table (\cite{Berger} pages 157--161).
\end{proof}

\subsection{Maximal reductive subalgebras}

We are particularly interested in reductive subalgebras which are
maximal. The following lemma provides the key to the reduction to
symmetric pairs.

\begin{lemma}\label{maximal subalgs}  
Let $\gf$ be a simple Lie algebra of real rank one 
and let $\hf\subset \gf$ be a maximal proper reductive subalgebra.
Then either $\hf$ is a symmetric subalgebra, or 
\begin{enumerate}
\item $\gf=\sp(1,n)$ and $\hf$ is conjugate to $\so(1,n)\times\sp(1)$, where $n>1$.
\item $\gf=\ff_4$ 
and $\hf$ is conjugate to $\su(1,2)\times \su(3)$.
\item $\gf=\ff_{4}$ 
and $\hf$ is conjugate to $\so(1,2)\times \gf_2$
\end{enumerate}
where $\gf_2$ denotes the compact real form of $\gf_2$. 

None of the pairs in {\rm (1)-(3)} are spherical.
\end{lemma}

\begin{proof} It is well known that the symmetric subalgebras of a simple
Lie algebra are maximal proper reductive subalgebras.

Reductive subalgebras are listed in \cite{Chen}, pages 276 and 284, 
and it is easily seen from these lists together with the list in Lemma
\ref{Berger list}
that only the subalgebras in (1)-(3) are maximal and non-symmetric.
The fact that these pairs are not spherical
will be proved in the following subsections.

\subsubsection{$(\gf,\hf)=(\sp(1,n),\so(1,n)\times\sp(1))$ is not spherical when $n>1$}

If $n=2$ then $\dim\hf=6$ and $\dim(\gf/\pf)=7$, so we may assume $n\ge 3$.
Like in the Lorentzian cases we identify for $G=\Sp(1,n)$ 
the flag variety $G/P=S^{4n-1}$ with a quadric in 
the quaternion projective space $\PP(V)$ where $V=\Hb^{n+1}$:
$$S^{4n-1}= \{ [z]\in \PP(V)\mid |z_0|^2 = |z_1|^2 + \ldots+|z_n|^2\}\, .$$
Here the action of $G$ on $V$ is from the left, and 
$$[z]=\{zh\mid h\in\Hb, h\neq 0\}\in \PP(V).$$
As a representation for $\SO_0(1,n)$, the space $V$ decomposes in four copies
of $V_0=\R^{n+1}$ with standard action. Hence the stabilizer in $H_n=\SO_0(1,n)$
of an element $v\in V$ is the stabilizer of four elements in $V_0$,
hence the centralizer of an at most four-dimensional subspace in $V_0$.
The centralizer in $\SO_0(1,n)$ of a four dimensional subspace of $\R^{n+1}$
is conjugate to the centralizer in $\SO_0(1,n-3)$ of a one dimensional
subspace of $\R^{n-2}$. Since all non-trivial orbits of 
$\SO_0(1,n-3)$ in $\R^{n-2}$ have codimension one, a simple computation shows
that the codimension in $\SO_0(1,n)$ of such a subgroup is $4n-6$. 
Hence orbits of $H_n$ in $V$ are at most of 
this dimension and orbits of
$H$ in $S^{4n-1}$ are at most of dimension $4n-3$.

\subsubsection{$(\ff_{4},\so(1,2)\times \gf_2)$ is not spherical}

Since $\dim G/P=15$, it suffices to show that the 
subgroup $G_2\subset K$ with Lie algebra $\gf_2$ has orbits in $K/M=G/P$
of dimension at most 11. Recall that $K=\Spin(9)$ and that we can
realize $K/M$ as the unit sphere in the 16-dimensional real spin 
representation $V_{16}$ of $K$. 
This representation decomposes for the standard
inclusions $\Spin(7)\subset\Spin(8)\subset\Spin(9)$ into a direct sum
of two copies of the spin representation $V_8$ of $\Spin(7)$.
Now $G_2$ is the isotropy subgroup of a spinor in $V_8$, and hence as a
$G_2$-representation
$$V_{16}= V_7 \oplus V_7 \oplus \R\oplus\R,$$
with a $7$-dimensional representation of $G_2$.
It follows that every orbit of $G_2$ lies in a product 
$R_1S^6\times R_2S^6\subset V_7\oplus V_7$
of spheres of radii $R_1,R_2\ge 0$.
Furthermore, the action of $G_2$ on $S^6\times S^6$ is not transitive as
the diagonal is invariant. 
Since $G_2$ is compact, we conclude that there are no open orbits 
on $S^6\times S^6$. This proves the claim.

\subsubsection{$(\ff_4,\su(2,1)\times \su(3))$ is not spherical}
Note that $\dim H=16$ and $\dim G/P=15$. 
Let us first collect a few facts about $\ff_4$. 
We refer to \cite{Chen} for more details.
Consider the Jordan algebra
\begin{equation}\label{Herm}
\mathrm{Herm}(3,\mathbb{O})_{2,1}= \left\{ x= \begin{pmatrix}  \alpha_1 & c_3 & - \bar{c}_2\\
\bar{c}_3 & \alpha_2 & c_1 \\ c_2  & - \bar{c}_1 & \alpha_3\end{pmatrix} \mid \alpha_i \in \R, c_i \in \mathbb{O}\right\}\,.
\end{equation}

The group $G$ of automorphisms of $W:=\mathrm{Herm}(3,\mathbb{O})_{2,1}$
is a real Lie group with Lie algebra $\ff_4$.  
Moreover the trace free elements $V:= W_{\mathrm{tr} =0}$ is 
an irreducible real representation for $G$ with a non-zero $K$-fixed vector.
Let $v_0\in V$ be a highest weight vector, then $P\cdot  v_0 = \R^+ v_0$
and we can realize the flag manifold as the image of 
$G\cdot v_0$ in $\PP(V)$. 
According to \cite{Chen}, p.\ 275,  
$\R^\times G\cdot v_0 =\mathcal{C}$,
where 
$$\mathcal{C}:= \{ x\in V\mid x^2=0, x\neq 0\},$$ 
and thus $G/P=\PP(\cc)$.

Note that $H=\SU(2,1) \times \SU(3)$ acts naturally on $V$. The factor 
$H_n=\SU(2,1)$ acts by matrix conjugation. Further, the automorphism group 
of $\mathbb{O}$ is $G_2$ and $H_c=\SU(3)$ is the subgroup which commutes with 
complex multiplication on $\mathbb{O}$.
We will show that every element $[x]\in\PP(\cc)$ has an at least 2-dimensional
stabilizer in $H$.

A straightforward matrix computation shows that if $x$ in (\ref{Herm})
satisfies $x^2=0$ and has trace zero, then up to multiplication by a real number
\begin{equation}\label{xone}
x=\begin{pmatrix} |c_2|^2&-\bar{c}_2\bar{c}_1&- \bar{c}_2\\ c_1c_2& |c_1|^2& c_1\\
c_2&-\bar{c_1}&-1
\end{pmatrix}
\end{equation}
with $|c_1|^2+|c_2|^2=1$.

In the sequel we decompose $\mathbb{O}=\C + \C^\perp$ and regard $\mathbb{O}_I:=\C^\perp$
as a complex vector space for the left action of $\C$. Then  
as a module for $\SU(3)$ it is equivalent with the standard complex
representation on $\C^3$.
Having said that we write the elements $x\in V$ as 
$$x=x_\C + x_I$$
where $x_\C \in i \su(2,1)\subset V$ and $x_I$ is of the form 
\begin{equation}\label{x_I}
x_I=\begin{pmatrix}  0 & c_3 & - \bar{c}_2\\
\bar{c}_3 & 0  & c_1 \\ c_2  & - \bar{c}_1 & 0 \end{pmatrix}
\end{equation}
with $c_1, c_2, c_3 \in \mathbb{O}_I.$
Note that this gives us a decomposition of $H$-modules. 

We see from (\ref{xone})
that 
$x_\C\neq 0$ for all $x\in \mathcal{C}$. Hence
the map
$$\PP(\mathcal{C}) \to \PP(i\su(2,1)),\ \  [x]\mapsto[x_\C]$$
is defined. 
As this is an open map, the image of an open $H$-orbit will be a non-empty open
set. Since the semisimple elements in $i\su(2,1)$ are dense, it suffices to consider
$x$ in (\ref{xone}) with $x_\C$ semisimple. 
If $x_\C \in i \su(2,1)$ is semisimple it is $\SU(2,1)$-conjugate to one of the following 
$$\begin{pmatrix}  \alpha_1 & 0 & 0\\ 0 & \alpha_2 & 0 \\ 0 & 0 & -\alpha_1 - \alpha_2\end{pmatrix}, 
\qquad 
\begin{pmatrix}  2\alpha & 0 & 0\\ 0 & -\alpha& \gamma i   \\ 0 & \gamma i  & -\alpha \end{pmatrix},
$$
where $\alpha_1, \alpha_2, \alpha, \gamma \in \R$. However, the second case does not conform with
(\ref{xone}).
Hence we may assume that $x=x_\C + x_I$ with $x_\C$ diagonal and with
$x_I$ as in (\ref{x_I}). It follows from (\ref{xone}) that $c_1c_2=\bar c_3$.

The fact that
$c_1c_2\in\bO_I$ implies that $c_1$ and $c_2$ are orthogonal elements.
Let $\mathbb{O}_I= \C j \oplus \C l \oplus \C n $ in the standard notation.
After application of $\SU(3)$ to $x_I$, it is no loss of generality to assume that $c_1=a j$ and $c_2=bl$ for some $a,b\in\R$.
Then $c_3= -ab n$ and
$$x_I=   \begin{pmatrix}  0 & -abn   &  bl \\
abn & 0 & a j \\ 
bl     & a j & 0\end{pmatrix}. $$
The diagonal torus $T<\SU(2,1)$ 
commutes with the diagonal matrix $x_\C$, and
embedded into 
$H$ via
$$t=\diag(t_1, t_2, t_3)\mapsto (t,t)\, $$
it also stabilizes $x_I$ -- note that for $z\in \C$ and $x\in \mathbb{O}_I$ one has $xz=\oline z x$. 
Hence the stabilizer of $x$ in $H$ has dimension at least 2.

This concludes the proof of Lemma \ref{maximal subalgs}.
\end{proof}

\section{Proofs}\label{sect 7}

All ingredients for the proofs have already been prepared.

\subsection{Proof of Theorem \ref{main thm2}}
Let $G$ be semisimple of real 
rank one and $H\subset G$ a connected reductive spherical subgroup.
As seen in Section \ref{simple} we may assume $G$ is simple.

Let $\hf'$ be a maximal proper reductive subalgebra which
contains $\hf$.
It follows from Lemma \ref{maximal subalgs} that $\hf'$ is symmetric,
and then it follows from Proposition \ref{prop} that $H$-orbits and
$H'$-orbits agree on $G/P.$

\subsection{Proof of Proposition \ref{th1} and Theorem \ref{main thm}} 

Theorem \ref{main thm2} implies the statement of Proposition \ref{th1} 
for reductive subgroups by the results of \cite{Wolf} or \cite{Mat}. 
By combining with Lemmas \ref{H in P} and \ref{reduction} 
we obtain the proposition for general subgroups. This also concludes
the proof of Theorem~\ref{main thm}.

\bigskip

\end{document}